\documentclass[11pt]{amsart}
\usepackage{amsfonts,amssymb}

\input xy
\xyoption{all}

\newcommand{\IC}{\mathbb C}
\newcommand{\IT}{\mathbb T}
\newcommand{\myodot}{{\mbox{\tiny $\odot$}}}
\newcommand{\moc}{\myodot}
\newcommand{\IN}{\mathbb N}

\newcommand{\cT}{\dot\IT}
\newcommand{\id}{\mathrm{id}}
\newcommand{\IR}{\mathbb R}

\newcommand{\e}{\varepsilon}

\newtheorem{theorem}{Theorem}
\newtheorem{proposition}{Proposition}
\newtheorem{lemma}{Lemma}

\newtheorem{corollary}{Corollary}
\theoremstyle{definition}
\newtheorem{definition}{Definition}
\newtheorem{example}{Example}

\title[Pontryagin duality for abelian inverse monoids]{Pontryagin duality between compact and\\ discrete abelian inverse monoids}
\author{Taras Banakh and Olena Hryniv}
\address{Instytut Matematyki, Uniwersytet Humanistyczno-Przyrodniczy Jana Kochanowskiego, Kielce, Poland}
\address{Department of Mathematics, Ivan Franko National University of Lviv, Ukraine}
\email{tbanakh@yahoo.com, olena\_hryniv@ukr.net}

\begin{document}
\begin{abstract} For a topological monoid $S$, the dual inverse monoid $S^\moc$ is the topological monoid of all 
identity preserving homomorphisms from $S$ to $\cT=\{z\in\IC:|z|\in\{0,1\}\}$, endowed with the compact-open topology.  A topological monoid $S$ is defined to be {\em reflexive} if the canonical homomorphism $\delta:S\to S^{\moc\moc}$ to the second dual inverse monoid is a topological isomorphism. We prove that a (compact or discrete) topological 
inverse monoid $S$ is reflexive (if and) only if $S$ is abelian and the idempotent semilattice $E=\{x\in S:xx=x\}$ of $S$ is zero-dimensional. For a discrete (resp. compact) topological monoid $S$ the dual inverse monoid $S^\moc$ is compact (resp. discrete). These results unify the Pontryagin-van Kampen Duality Theorem for abelian groups and the Hofmann-Mislove-Stralka Duality Theorem for zero-dimensional topological semilattices.
\end{abstract}
\maketitle

In this paper we unify the duality theories for abelian topological groups 
\cite{Morris} and topological semilattices \cite{HMS} and develop a duality theory for abelian topological inverse monoids.

Let us recall that a semigroup $S$ is 
\begin{itemize}
\item a {\em monoid} if it has a two-sided unit 1;
\item an {\em idempotent semigroup} if $xx=x$ for all $x\in S$;
\item {\em inverse} if for each $x\in S$ there is a unique element $x^{-1}\in S$ such that $xx^{-1}x=x$ and $x^{-1}xx^{-1}=x^{-1}$;
\item {\em Clifford}, if $S$ each element of $S$ lies in a subgroup of $S$;
\item {\em abelian} if $xy=yx$ for all $x,y\in S$;
\item a {\em semilattice} if $S$ is an abelian idempotent semigroup.
\end{itemize}

It is well-known that for an inverse semigroup $S$ the set of idempotents $E=\{e\in S:ee=e\}$ is a semilattice. This semilattice will be called {\em the idempotent semilattice} of $S$. For each idempotent $e\in S$ the set
$$H_e=\{x\in S:xx^{-1}=e=x^{-1}x\}$$is the {\em  maximal subgroup} of $S$ that contains the idempotent $e$. If $S$ is a Clifford inverse semigroup, then $S=\bigcup_{e\in E}H_e$ is the union of maximal groups and the map $\pi:S\to E$, $\pi:x\mapsto xx^{-1}=x^{-1}x$, is a homomorphism. It is known that an inverse semigroup $S$ is Clifford if and only if $xx^{-1}=x^{-1}x$ for all $x\in S$. An abelian semigroup is Clifford if and only if it is inverse, see \cite[II.2]{InvS}.  

A {\em topological inverse semigroup} is an inverse semigroup $S$ endowed with a topology making the multiplication $S^2\to S$, $(x,y)\mapsto xy$, and the inversion $S\to S$, $x\mapsto x^{-1}$, continuous. A {\em topological inverse monoid} is a topological inverse semigroup with a two-sided unit 1. All topological spaces considered in this paper are Hausdorff.

The class of topological inverse monoids includes all topological groups and all topological semilattices with a unit. An important example of an abelian topological inverse monoid is the compact subset
$$\dot\IT=\{z\in\IC:|z|\in\{0,1\}\}$$of the complex plane endowed with the operation of multiplication of complex numbers. 
The inverse monoid $\dot\IT$ can be written as the union $\dot\IT=\IT\cup\{0,1\}$ of the topological group $$\IT=\{z\in\IC:|z|=1\}$$ and the topological semilattice $\{0,1\}$ endowed with the operation of minimum.

A function $h:S\to M$ between monoids is called a {\em monoid homomorphism} if $h(1)=1$ and $h(xy)=h(x)h(y)$ for all $x,y\in S$.
 
For a topological monoid $S$ by $S^{\myodot}$ we denote the space of continuous monoid homomorphisms $S\to \dot\IT$, endowed with the compact-open topology.
Endowed with the operation of pointwise multiplication of homomorphisms, $S^{\myodot}$ is an abelian topological inverse monoid. The monoid $S^{\myodot}$ is called {\em the dual inverse monoid} of the topological monoid $S$. 
It contains a closed subgroup $$S^\circ=\{h\in S^{\myodot}:h(S)\subset\IT\}$$ called the {\em dual group} of $S$ and a closed idempotent submonoid
$$S^{\mathbf:}=\{h\in S^{\myodot}:h(S)\subset\{0,1\}\}$$  called the {\em dual idempotent monoid} of $S$. 
It is easy to see that $S^:$ coincides with the idempotent semilattice of the inverse semigroup $S^\moc$.

Observe that
$$S^\myodot=\begin{cases}
S^\circ&\mbox{if $S$ is a group};\\
S^:&\mbox{if $S$ is an idempotent monoid.}
\end{cases}
$$

The following property of the dual inverse monoid is immediate:

\begin{proposition} 
For any topological monoid $S$ the dual inverse monoid $S^\moc$ 
is an abelian topological inverse monoid whose idempotent semilattice is zero-dimensional and coincides with the dual idempotent monoid $S^:$.
\end{proposition}

The zero-dimensionality of the dual idempotent monoid $S^:$ follows from the fact that it embeds into the  function space $C_k(S,\{0,1\})$ endowed with the compact-open topology. The latter function space is zero-dimensional.
We recall that a topological space $X$ is {\em zero-dimensional} if closed-and-open subsets form a base of the topology
of $X$.

Any continuous monoid homomorphism $h:S\to M$ between topological monoids induces a continuous monoid homomorphism $$h^\moc:M^\moc\to S^\moc,\;\;h^\moc:f\mapsto f\circ h$$ called the {\em dual homomorphism} of $h$. Now we see that $(\cdot)^\moc$ is a contravariant functor from the category of topological monoids and their continuous monoid homomorphisms to the category of abelian topological inverse monoids and their monoid homomorphism.

%The functor $(\cdot)^\moc$ nicely commutes with products. Namely, it is easy to check that for any topological monoids $X,Y$ the function 
%$$\Pi:X^\moc\times Y^\moc\to (X\times Y)^\moc,\;\;\Pi:(\varphi,\psi)\mapsto\varphi\times\psi,\;\mbox{ where }\;(\varphi\times\psi)(x,y)=\varphi(x)\cdot\psi(y)$$
%is a topological isomorphism. 

For a topological monoid $S$ the map $$\delta_S:S\to S^{\myodot\myodot},\;\delta_S:x\mapsto \delta_S(x)\mbox{ where  $\delta_S(x):h\mapsto h(x)$ for $h\in S^\moc$}$$ is a monoid homomorphism called the {\em canonical homomorphism} of $S$ into its second dual inverse monoid $S^{\myodot\myodot}$. 

By analogy we can define the canonical homomorphisms 
$$\delta_S:S\to S^{\circ\circ}\mbox{ \ and \ }\delta_S:S\to S^{::}$$of $S$ into its second dual group and second dual semilattice. If $S$ is clear from the context, we shall omit the subscript and write $\delta$ instead of $\delta_S$.
Now we define a principal notion of this paper.

\begin{definition} A topological monoid $S$ is called {\em reflexive} if the canonical homomorphism $\delta_S:S\to S^{\myodot\myodot}$ is a topological isomorphism. 
\end{definition}
If $S$ is a group, then $S^{\myodot\myodot}=S^{\circ\circ}$ and thus $S$ is reflexive if and only if the canonical homomorphism $\delta:S\to S^{\circ\circ}$ to the second dual group is a topological isomorphism, which means that $S$ is reflexive in the usual sense \cite{Refl}. A similar fact holds if $S$ is an idempotent monoid: since $S^{\myodot\myodot}=S^{::}$, $S$ is reflexive if and only if the canonical homomorphism $\delta:S\to S^{::}$ is a topological isomorphism.

Now we are able to formulate the duality theorems for abelian topological groups and abelian idempotent topological  monoids. The following  Duality Theorem is classical and can be found in  \cite{Morris}.

\begin{theorem}[Pontryagin - van Kampen]\label{group}
\begin{enumerate}
\item A topological group $G$ is reflexive (if and) only if $G$ is abelian (and locally compact).
\item If a topological group $G$ is compact (resp. discrete), then its dual group $G^\circ=G^\myodot$ is discrete (resp. compact).
\end{enumerate}
\end{theorem}

A similar Duality Theorem for zero-dimensional topological idempotent monoids was proved by Hofmann, Mislove, and Stralka \cite{HMS}.

\begin{theorem}[Hofmann - Mislove - Stralka]\label{semilat} 
\begin{enumerate}
\item A topological idempotent monoid $S$ is reflexive (if and) only if $S$ is abelian and zero-dimensional (and $S$ is either discrete or compact).
\item If a topological idempotent monoid $S$ is compact (resp. discrete), then its dual idempotent monoid $S^{:}=S^\myodot$ is discrete (resp. compact).
\end{enumerate}
\end{theorem}

Now our strategy is to unify these two Duality Theorems and prove a Duality Theorem for abelian topological inverse monoids. First we generalize the second parts of Duality Theorems~\ref{group} and \ref{semilat}:

\begin{theorem}\label{t3} If a topological monoid $S$ is discrete (resp. compact), then its dual inverse monoid $S^\myodot$ is compact (resp. discrete).
\end{theorem}    

Next, we characterize the reflexivity in the classes of compact or discrete topological monoids.

\begin{theorem}\label{t4} A compact or discrete topological monoid $S$ is reflexive if and only if $S$ is an abelian topological inverse monoid with zero-dimensional idempotent semilattice $E$.
\end{theorem}

Theorems~\ref{t3} and \ref{t4} imply the following duality between compact and discrete reflexive topological inverse monoids.

\begin{corollary} A reflexive topological monoid $S$ is compact (reps. discrete) if and only if its dual inverse monoid $S^\moc$ is discrete (resp. compact).
\end{corollary} 

It turns out that under some restrictions on the maximal semilattice and maximal groups, the reflexivity of an abelian topological inverse monoid is equivalence to its compactness.

\begin{theorem}\label{refcom} Let $S$ be an abelian topological inverse monoid with compact zero-dimensional idempotent semilattice $E$ and compact maximal groups $H_e$, $e\in E$. The topological monoid $S$ is reflexive if and only if $S$ is compact.
\end{theorem}

Finally, we present two examples showing that (in contrast to Theorem~\ref{group}) Theorem~\ref{t4} 
 cannot be generalized to locally compact topological inverse monoids.

\begin{example} Consider the convergent sequence $E_1=\{1\}\cup\{1-\frac1n:n\in\IN\}\subset\IR$ endowed with the semilattice operation $x\cdot y=\max\{x,y\}$, and the 2-element cyclic subgroup $C_2=\{-1,1\}$  of $\IT$. The product $E_1\times C_2$ is an abelian compact inverse monoid and $S=(E_1\times C_2)\setminus\{(1,-1)\}$ is a submonoid of $E_1\times C_2$. It is clear that  $S$ is locally compact and not compact, the maximal semilattice $E=E_1\times\{1\}$ of $S$ is compact and all maximal subgroups $H_e=S\cap(\{x\}\times C_2)$, $e=(x,1)\in E$, are compact. By Theorem~\ref{refcom} the (locally compact) topological monoid $S$ is not reflexive.
\end{example}

\begin{example} Consider the convergent sequence  $E_0=\{0\}\cup\{\frac1n:n\in\IN\}$  endowed with the semilattice operation $x\cdot y=\min\{x,y\}$.
On the product $S=E_0\times C_2$ consider the commutative semigroup operation
$$(t,x)\cdot(s,y)=\begin{cases}
(t,x)&\mbox{if $t<s$,}\\
(s,y)&\mbox{if $t>s$,}\\
(t,xy)&\mbox{if $t=s$}.
\end{cases}
$$
Endow $S$ with the topology that induces the original topology on $E_0\times\{1\}$ and the discrete topology on $E_0\times\{-1\}$. It is clear that the obtained topological space $S$ is locally compact but not compact. Also, the idempotent semilattice $E=E_0\times\{1\}$ is compact and all maximal subgroups $H_e=\{x\}\times C_2$, $e=(x,1)\in E$, are compact.
By Theorem~\ref{refcom}, the (locally compact) topological monoid $S$ is not reflexive. It can be shown that the canonial homomorphism $\delta:S\to S^{\moc\moc}$ is bijective and the second dual monoid $S^{\moc\moc}$ is compact and can be identified with $S$ endowed with the product topology of $E_0\times C_2$.
\end{example}

\section{Some Lemmas}\label{s:aux}

In this section we prove several lemmas that will be used in the proof of Theorem~\ref{t4}.

First we find a topological condition on the topological monoid $S$ guaranteeing that the canonical homomorphism $\delta_S:S\to S^{\moc\moc}$ is continuous.
We recall that a topological space $X$ is a {\em $k$-space} if a subset $F\subset X$ is closed in $X$ if and only for any compact subset $K\subset X$ the intersection $K\cap F$ is closed in $K$. The class of $k$-spaces includes all regular locally compact spaces and all metrizable spaces, see \cite[\S3.3]{En}.

\begin{lemma}\label{cont} If a topological monoid $S$ is a $k$-space, then the canonical homomorphism $\delta_S:S\to S^{\moc\moc}$ is continuous.
\end{lemma}

\begin{proof} By definition, the compact-open topology on the second dual inverse monoid $S^{\moc\moc}$ is generated by the sub-base consisting of the sets
$$[K,U]=\{\mu\in S^{\moc\moc}:\mu(K)\subset U\}$$where $K$ is a compact subset of $S^\moc$ and $U$ is an open subset of $\cT$. So, the continuity of $\delta_S$ will follow as soon as we check that for any compact $K\subset S^\moc$ and open $U\subset\cT$ the set $$\delta_S^{-1}([K,U])=\{x\in S:\forall h\in K\;h(x)\in U\}$$ is open in $S$. Given any point $x\in \delta_S^{-1}([K,U])$, we should find a neighborhood $O(x)\subset\delta_S^{-1}([K,U])$. Since the calculation map $c_x:S^\moc\to \cT$, $c_x:h\mapsto h(x)$, is continuous, the set $K(x)=\{h(x):h\in K\}\subset U$ is compact. Consequently, $\e=\inf\{|z-y|:z\in K(x),\;y\in\cT\setminus U\}$ is positive. 

 By the Ascoli Theorem 8.2.10 \cite{En}, the compact subset $K$ of $S^\moc$ is equicontinuous. Consequently, the point $x$ has a neighborhood $O(x)\subset S$ such that $|h(x)-h(x')|<\e$ for all $x'\in O(x)$ and $h\in K$. Now the choice of $\e$ guarantees that $h(x')\in U$ and thus $O(x)\subset \delta_S^{-1}([K,U])$ witnessing that the latter set is open in $S$.
\end{proof}

For an inverse semigroup $S$ its idempotent semilattice will be denoted by $E(S)$ or just $E$ if $S$ is understood from the context.

\begin{lemma}\label{l1} Let $E$ be the idempotent semilattice of a Clifford topological inverse monoid $S$ and $i:E\to S$ be the identity inclusion. The dual homomorphism $i^: :S^{\mathbf:}\to E^:$, $i^::h\mapsto h\circ i=h|E$, is a topological isomorphism.
\end{lemma}

\begin{proof} The inverse to $i^:$ is the dual map 
$$\pi^::E^:\to S^:,\;\pi^::h\mapsto h\circ\pi,$$
to the homomorphism $\pi:S\to E$, $\pi:x\mapsto xx^{-1}=x^{-1}x$.
The equality $i^:\circ\pi^:=\id_{E^:}$ follows from $\pi\circ i=id_E$. To show that $\pi^:\circ i^:=\id_{S^:}$, take any homomorphism $h\in S^:$ and a point $x\in S$. 
Observe that $h(S)\subset\{0,1\}$ implies that 
$$h\circ \pi(x)=h(xx^{-1})=h(x)h(x)^{-1}=h(x)$$ and thus $h=h\circ \pi=\pi^:\circ i^:(h)$.
\end{proof}

Since for an inverse semigroup $S$ the idempotent semilattice of $S^\moc$ coincides with $S^:$, Lemma~\ref{l1} implies: 

\begin{corollary}\label{ES} For any Clifford topological inverse monoid $S$ and the identity inclusion $i:E(S)\to S$ the dual homomorphism $i^\moc:S^\moc\to E(S)^{\moc}$ maps isomorphically the idempotent semilattice $E(S^\moc)$ of $S^\moc$ onto $E(S)^\moc=E(S)^:$.
\end{corollary}

\begin{lemma}\label{dixmier} For any topological monoid $S$ the composition $\delta_S^\moc\circ \delta_{S^\moc}$ of the homomorphisms
$$\xymatrix{
S^\moc\ar[r]^{\delta_{S^\moc}}&S^{\moc\moc\moc}\ar[r]^{\delta^\moc_S}&S^\moc}
$$
equals the identity homomorphism of $S^\moc$.
\end{lemma}

\begin{proof} Given any homomorphism $h\in S^\moc$, we need to check that $\delta_S^\moc(\delta_{S^\moc}(h))=h$.
Let $H=\delta_{S^\moc}(h)$ and note that for any $\mu\in S^{\moc\moc}$
$H(\mu)=\mu(h)$ by the definition of the canonical homomorphism $\delta_{S^\moc}:S^\moc\to S^{\moc\moc\moc}$. In particular, for any $x\in S$ we get $$H(\delta_S(x))=(\delta_S(x))(h)=h(x)$$ where the latter equality follows from the definition of $\delta_S$. By the definition of the dual homomorphism $\delta_S^\moc:S^{\moc\moc\moc}\to S^\moc$, we get
$$\delta_S^\moc(\delta_{S^\moc}(h))(x)=\delta_S^\moc(H)=H(\delta_{S}(x))=h(x)$$and hence 
$\delta_S^\moc(\delta_{S^\moc}(h))=h$.
\end{proof}

Let $S$ be a Clifford topological inverse semigroup and $E$ be its idempotent semilattice. For each element $e\in E$ consider the upper cone ${\uparrow}e=\{f\in E:f\ge e\}$. Here we endow $E$ with the partial order: $e\le f$ iff $ef=e$.
It is clear that for each $e\in E$ the upper cone ${\uparrow}e$ is closed in $E$. 
An idempotent $e\in E$ will be called {\em locally minimal} if the upper cone ${\uparrow}e$ is open in $E$.
This is equivalent to saying that $e$ is the smallest element of some neighborhood of $e$ in $E$. The following lemma can be easily derived from Theorem II.1.12 of \cite{HMS}.

\begin{lemma}\label{locmin} Let $S$ be a compact zero-dimensional topological semilattice. For any point $e\in E$ and a neighborhood $O(e)\subset E$ of $e$ there is a locally minimal element $e'\in O(e)$ such that $e'\le e$.
\end{lemma}

Finally we describe a very useful construction of extension of a homomorphism $h:H_e\to \IT$ from a maximal subgroup $H_e$ of a Clifford inverse semigroup $S$ to a homomorphism $\bar h:S\to\cT$.

For each locally minimal idempotent $e\in E$ consider the homomorphism 
 $\Lambda_e:H_e^\circ\to S^\moc$ assigning to each homomorphism $h:H_e\to\IT$ the monoid homomorphism $\Lambda_eh\in S^\moc$ defined by
$$\Lambda_e h:x\mapsto \begin{cases}
h(xe)&\mbox{if $xe\in H_e$}\\
0&\mbox{otherwise}.
\end{cases}
$$
The local minimality of $e$ guarantees that so-defined homomorphism $\Lambda_eh$ is continuous.

\section{Proof of Theorem~\ref{t3}} We divide the proof of Theorem~\ref{t3} into two lemmas.

\begin{lemma}\label{l:dc} If $S$ is a discrete topological monoid, then its dual inverse monoid $S^\moc$ is compact.
\end{lemma}

\begin{proof} Since $S$ is discrete, the compact-open topology on the dual inverse monoid $S^\moc$ coincides with the topology inherited from the Tychonov product $\cT^S$ under the identity inclusion $S^\moc\to \cT^S$. Being a closed subset of the compact Hausdorff space $\cT^S$, the dual inverse monoid
$$S^\moc=\{h\in \cT^S:\mbox{$h(1)=1$ and $\forall x,y\in S$ \ \ $h(xy)=h(x)h(y)$}\}\subset\cT^S$$ of $S$ is compact.
\end{proof}

The ``compact'' part of Theorem~\ref{t3} is a bit more difficult.

\begin{lemma}\label{l:cd} If $S$ is a compact topological monoid, then its dual inverse monoid $S^\moc$ is discrete.
\end{lemma}

\begin{proof} Consider the homomorphism 
$$\rho:S\to \cT^{S^\moc},\;\;\rho:x\mapsto \big(f(x)\big)_{f\in S^\moc}.$$ By the continuity of $\rho$, the image $\tilde S=\rho(S)$ of $S$ is a compact submonoid of the topological inverse monoid $\cT^{S^\moc}$. Since for each element $x\in\tilde S$ the inverse $x^{-1}$ of $x$ in the inverse semigroup $\cT^{S^{\moc}}$ lies in the closure of the set $\{x^n:n\in\IN\}\subset\tilde S$, $\tilde S$ is an inverse subsemigroup of $\cT^{S^\moc}$.

It follows from the definition of $\tilde S$ that for each continuous monoid homomorphism $h:S\to\cT$ there is a unique continuous monoid homomorphism $\tilde h:\tilde S\to\cT$ such that $h=\tilde h\circ\rho$. This implies that the dual inverse monoids $S^\moc$ and $\tilde S^\moc$ are topologically isomorphic. 
The following lemma implies that for the compact abelian topological inverse monoid $\tilde S$ the dual inverse monoid $\tilde S^\moc=S^\moc$ is discrete.
\end{proof}

\begin{lemma}\label{disc} Let $S$ be an abelian topological inverse monoid with compact idempotent semilattice $E$. If for each locally minimal idempotent $e\in E$ the maximal subgroup $H_e$ is compact, then the dual inverse monoid $S^\moc$ is discrete.
\end{lemma}
 
\begin{proof} To see that $S^\moc$ is discrete, fix any homomorphism $h\in S^\moc$ and consider the closed-and-open subsemilattices $E_i=E\cap h^{-1}(i)$, $i\in\{0,1\}$, of the compact semilattice $E$. Being compact, the subsemilattice $E_1$ has the smallest element $e$, which is locally minimal in $E$ as $E_1={\uparrow}e$ is open in $E$. By our assumption, 
the maximal subgroup $H_e=\{x\in S:xx^{-1}=e\}$ is compact. By the Pontryagin-van Kampen Duality Theorem~\ref{group}, the dual group $H_{e}^\circ$ is discrete. Consequently, the singleton $\{h|H_{e}\}$ is open in $H_{e}^\circ$. It follows that $O(h)=\{f\in S^\moc:f|E_0\subset\{0\},\;f|H_e=h|H_e\}$ is an open neighborhood of $h$ in the compact-open topology of $S$. It remains to check that $U=\{h\}$.
Indeed, take any homomorphism $f\in U$ and a point $x\in S$. If $xx^{-1}\in E_0$, then $0=f(xx^{-1})=f(x)\cdot(f(x))^{-1}$ and hence $f(x)=0=h(x)$. If $xx^{-1}\in E_1$, then $xe\in H_e$ and thus
$$f(x)=f(x)\cdot 1=f(x)h(e)=f(x)f(e)=f(xe)=h(xe)=h(x)\cdot1=h(x),$$ witnessing that $f=h$.
\end{proof}

\section{Proof of Theorem~\ref{t4}}

We divide the proof of Theorem~\ref{t4} into two lemmas.

\begin{lemma}\label{l:dr} Each abelian discrete topological inverse monoid $S$ is reflexive.
\end{lemma}

\begin{proof} We need to show that the canonical homomorphism $\delta_S:S\to S^{\moc\moc}$ is a topological isomorphism. By Theorem~\ref{t3}, the dual inverse monoid $S^{\moc}$ is compact and the second dual $S^{\moc\moc}$ is discrete. So, it suffices to check that $\delta_S$ is bijective.   

Consider the homomorphism $\pi:S\to E$, $\pi:x\mapsto xx^{-1}=x^{-1}x$ that retracts $S$ onto its idempotent semilattice $E$. 

To show that the canonical homomorphism $\delta_S$ is injective, take any two distinct points $x,y\in S$ and consider the idempotents $\pi(x)$ and $\pi(y)$. If these idempotents are distinct, then by the Hofmann-Mislove-Stralka Duality Theorem \ref{semilat}, there is a monoid homomorphism $h:E\to\{0,1\}$ such that $h(\pi(x))\ne h(\pi(y))$. Then the homomorphism $h\circ\pi$ belongs to $S^\moc$ and
$$\delta_S(x)(h\circ\pi)=h\circ\pi(x)\ne h\circ\pi(y)=\delta_S(y)(h\circ\pi),$$witnessing that $\delta_S(x)\ne\delta_S(y)$.

Next, assume that $\pi(x)=\pi(y)$. In this case the points $x,y$ belong to the maximal subgroup $H_e$ of the idempotent $e=xx^{-1}=yy^{-1}$. By the Pontryagin-van Kampen Duality Theorem~\ref{group}, there is a group homomorphism $h:H_e\to\IT$ such that $h(x)\ne h(y)$. Since $E$ is discrete, the element $e$ is locally minimal and hence it is legal to consider the extension $\Lambda_eh\in S^\moc$ of $h$ discussed in Section~\ref{s:aux}. For the homomorphism $\Lambda_eh$ we get:
$$\delta_S(x)(\Lambda_eh)=\Lambda_eh(x)=h(x)\ne h(y)=\Lambda_eh(y)=\delta_S(y)(\Lambda_eh),$$
 witnessing that $\delta_S(x)\ne\delta_S(y)$. After considering these two cases, we conclude that the canonical homomorphism $\delta$ is injective.
\smallskip

Next, we show that $\delta:S\to S^{\moc\moc}$ is surjective. Fix any element $\mu\in S^{\moc\moc}$.  The homomorphism $\pi:S\to E$, $\pi:x\mapsto xx^{-1}$, induces the dual homomorphism $\pi^\moc:E^\moc\to S^\moc$ 
 and the second dual homomorphism
$\pi^{\moc\moc}:S^{\moc\moc}\to E^{\moc\moc}$. By Hofmann-Mislove-Stralka Theorem~\ref{semilat}, the discrete idempotent monoid $E$ is reflexive and thus is topologically isomorphic to its second dual idempotent monoid $E^{::}=E^{\moc\moc}$. Consequently, the element $\pi^{\moc\moc}(\mu)$ coincides with the image $\delta_E(e)$ of some idempotent $e\in E$ under the canonical homomorphism $\delta_E:E\to E^{\moc\moc}=E^{::}$. It follows that for any homomorphism $h\in E^:$ we get
$\mu(h\circ\pi)=h(e)$.

Now consider the maximal subgroup $H_{e}\subset S$. Since $e$ is locally minimal, it is legal to consider the continuous homomorphism $\Lambda_e:H_e^\circ\to S^\moc$ discussed in Section~\ref{s:aux}. 
Then the composition $\mu\circ\Lambda_e:H_e^\circ\to \cT$ is a continuous semigroup homomorphism. 
We claim that $\mu\circ\Lambda(H_e^\circ)\subset\IT$. Since $\mu\circ\Lambda(H_e^\circ)$ is a subgroup of $\cT$, it suffices to check that $\mu\circ\Lambda_e(\bold 1)=1$ for the trivial homomorphism $\bold 1:H_e\to\{1\}\subset\cT$.
Since $\Lambda_e\bold 1=(\Lambda_e\bold 1)\circ\pi$, we get
$$\mu(\Lambda_e\bold 1)=\mu(\Lambda_e\bold 1\circ\pi)=(\Lambda_e\bold 1)(e)={\mathbf 1}(e)=1.$$Therefore, $\mu\circ\Lambda_e:H_e^\circ\to\IT$ belongs to the second dual group $H_e^{\circ\circ}$.

 By the Pontryagin-van Kampen Duality Theorem~\ref{group}, the discrete abelian group $H_{e}$ is reflexive and thus
the homomorphism $\mu\circ\Lambda_e$ is equal to $\delta_{H_e}(z)$ for some point $z\in H_e$.
We claim that $\mu=\delta_S(z)$. For this we need to check that $\mu(h)=h(z)$ for each homomorphism $h\in S^\moc$. 
Two cases are possible. 

0) $h(z)=0$. In this case $h(e)=h\circ\pi(z)=h(z)\cdot h(z)^{-1}=0$ and thus $\mu(h\circ\pi)=h(e)=0$. Since $h=h\cdot (h\circ\pi)$, we conclude that
$$\mu(h)=\mu(h\cdot (h\circ\pi))=\mu(h)\cdot \mu(h\circ\pi)=0=h(z).$$

1) $h(z)\ne 0$. In this case $h(e)=h\circ\pi(z)=h(z)\cdot h(z)^{-1}=1$ and thus $\mu(h\circ\pi)=h(e)=1$. Also $h|H_e:H_e\to \IT$ is a well-defined homomorphism on the group $H_e$. Let $g=\Lambda_e(h|H_e)\in S^\moc$ and observe that for each $x\in S$ with $\pi(x)\ge e$ we get $g(x)=h(ex)=h(e)h(x)=1\cdot h(x)=h(x)$. Consequently, $h\cdot (g\circ\pi)=g\cdot (h\circ\pi)$. By the choice of $z$, we get $\mu(g)=h(z)\ne 0$ and $\mu(g\circ\pi)=\mu(g\cdot g^{-1})=h(z)\cdot h(z)^{-1}=1$. Then 
$$\mu(h)=\mu(h)\cdot 1=\mu(h)\cdot \mu(g\circ\pi)=\mu(h\cdot (g\circ\pi))=\mu((h\circ \pi)\cdot g)=\mu(h\circ\pi)\cdot \mu(g)=1\cdot h(z).$$
This completes the proof of the surjectivity of the canonical homomorphism $\delta:S\to S^{\moc\moc}$.
\end{proof}

\begin{lemma} Each compact topological inverse monoid $S$ with zero-dimensional idempotent semilattice $E$ is reflexive.
\end{lemma}

\begin{proof}
 We need to show that the canonical homomorphism $\delta_S:S\to S^{\moc\moc}$ is a topological isomorphism. By Lemma~\ref{cont}, $\delta_S$ is continuous. Since $S$ is compact, it suffices to check that $\delta_S$ is bijective.   
 
Take any two distinct points $x,y\in S$. If the idempotents $xx^{-1},yy^{-1}$ are distinct, then repeating the argument from the Lemma~\ref{l:dr}, we can construct a continuous monoid homomorphism $h:E\to\{0,1\}$ such that $\delta_S(x)(h\circ\pi)\ne \delta_S(y)(h\circ \pi)$. Next, assume that $xx^{-1}=yy^{-1}$. 
Fix two disjoint neighborhoods $O(x)$ and $O(y)$ of the points $x,y$ in $S$ and by the continuity of the semigroup operation, find a neighborhood $O(\tilde e)\subset E$ of the idempotent $\tilde e=xx^{-1}=yy^{-1}$ such that $xO(\tilde e)\subset O(x)$ and $yO(\tilde e)\subset O(y)$. By Lemma~\ref{locmin}, the neighborhood $O(\tilde e)$ contains a locally minimal idempotent $e\le\tilde e$. The choice of the neighborhood $O(\tilde e)$ guarantees that $ex\ne ey$. Then by the Pontryagin-van Kampen Duality Theorem~\ref{group}, we can find a continuous group homomorphism $h:H_e\to\IT$ such that $h(ex)\ne h(ey)$. Extend $h$ to the continuous monoid homomorphism $\Lambda_eh\in S^\moc$.
Then $$\delta_S(x)(\Lambda_e h)=\Lambda_e h(x)=h(ex)\ne h(ey)=\Lambda_e h(y)=\delta_S(y)(\Lambda_eh)$$ and thus $\delta_S(x)\ne\delta_S(y)$, witnessing that the homomorphism $\delta:S\to S^{\moc\moc}$ is injective.

It remains to prove the surjectivity of the canonical homomorphism $\delta_S:S\to S^{\moc\moc}$. Assuming that $\delta$ is not surjective, find a point $x\in S^{\moc\moc}\setminus\delta(S)$ and consider the idempotent $\tilde e=xx^{-1}$ of $S^{\moc\moc}$. By Lemma~\ref{ES}, the idempotent semilattice $E(S^{\moc\moc})$ of the second dual inverse monoid $S^{\moc\moc}$ can be identified with the second dual semilattice $E^{::}$ of $E=E(S)$. By Hofmann-Mislove-Stralka Duality Theorem~\ref{semilat}, the topological semilattice $E$, being compact, is reflexive and thus can be identified with its second dual idempotent monoid $E^{::}$. Consequently, the idempotent semilattice of $S^{\moc\moc}$ coincides with $\delta_S(E)$.

Since $x\tilde e=x\notin\delta(S)$ and the set $\delta(S)$ is closed in $S^{\moc\moc}$, the idempotent $\tilde e$ has a neighborhood $O(\tilde e)\subset E(S^{\moc\moc})$ such that $ex\notin\delta(S)$ for any idempotent $e\in O(\tilde e)$.  By Lemma~\ref{locmin}, $O(\tilde e)$ contains a locally minimal idempotent $e\le \tilde e$. Let $\tilde H_e$ be the maximal subgroup of $S^{\moc\moc}$ that contains the idempotent $e$. It follows that $ex\in\tilde H_e\setminus \delta(S)$. Observe that $\tilde H_e\cap\delta(S)$ is a closed subgroup of $\tilde H_e$. Using the Pontryagin-van Kampen Duality Theorem~\ref{group}, we can find a continuous 
homomorphism $h:\tilde H_e\to\IT$ such that $h(\tilde H_e\cap\delta(S))\subset\{1\})$ while $h(ex)\ne 1$. Extend the homomorphism $h$ to the continuous homomorphism $\tilde h=\Lambda_eh:S^{\moc\moc}\to\cT$. The extended homomorphism has the property $\tilde h(\delta(S))\subset\{0,1\}$ and thus $\tilde h|\delta(S)=\tilde h\circ\pi|\delta(S)$ while $\tilde h\ne\tilde h\circ \pi$.

Now consider the dual homomorphism $\delta_S^\moc:S^{\moc\moc\moc}\to S^\moc$. 
It follows that 
$$\delta_S^\moc(\tilde h)=\tilde h\circ\delta_S=\tilde h\circ\pi\circ\delta_S=\delta_S^\moc(\tilde h\circ\pi)$$and hence $\delta^\moc_S$ is not injective. 

On the other hand, by Theorem~\ref{t3}, the dual inverse monoid $S^\moc$ of $S$ is discrete and hence reflexive according to Lemma~\ref{l:dr}. Then the canonical homomorphism $\delta_{S^\moc}:S^\moc\to S^{\moc\moc\moc}$ is a topological isomorphism. By Lemma~\ref{dixmier},
the composition $\delta_S^\moc\circ\delta_{S^\moc}$ coincides with the identity homomorphism of $S^\moc$. Since $\delta_{S^\moc}$ is a topological isomorphism, so is the dual homomorphism $\delta_{S}^\moc:S^{\moc\moc\moc}\to S^\moc$. 
In particular, $\delta^\moc_S$ is injective, which is a desired contradiction.
\end{proof}

\section{Proof of Theorem~\ref{refcom}}

 Let $S$ be an abelian topological monoid with compact zero-dimensional idempotent semilattice $E$ and compact maximal groups $H_e$, $e\in E$. By Lemma~\ref{disc}, the dual inverse monoid $S^\moc$ is discrete and by Theorem~\ref{t3}, the second dual monoid $S^{\moc\moc}$ is compact. If $S$ is reflexive, then $S$ is compact, being topologically isomorphic to $S^{\moc\moc}$. If $S$ is compact, then it is reflexive by Theorem~\ref{t4}.


\begin{thebibliography}{}

\bibitem{Refl} L.~Aussenhofer, 
Contributions to the duality theory of abelian topological groups and to the theory of nuclear groups.  
Dissert. Math.  {\bf 384} (1999), 113 pp.

\bibitem{En} R.~Engelking, General topology, Heldermann Verlag, Berlin, 1989.

\bibitem{HMS} K.~Hofmann, M.~Mislove, A.~Stralka,  The Pontryagin duality of compact $0$-dimensional semilattices and its applications, Springer-Verlag, Berlin-New York, 1974.

\bibitem{Morris} S.~Morris, Pontryagin duality and the structure of locally compact abelian groups, Cambridge Univ. Press, Cambridge-New York-Melbourne, 1977.

\bibitem{InvS} M.~Petrich, Inverse semigroups, Wiley-Intersci. Publ., New York, 1984. 

\end{thebibliography}
\end{document}